\documentclass{article}
\usepackage{amsmath,amssymb,amsthm,eucal,graphicx}

\newtheorem{theorem}{Theorem}

\newtheorem{definition}{Definition}
\newtheorem{corollary}{Corollary}
\newtheorem{example}{Example}

\numberwithin{theorem}{section}
\numberwithin{definition}{section}
\numberwithin{lemma}{section}
\numberwithin{corollary}{section}
\numberwithin{equation}{section}
\numberwithin{proposition}{section}
\numberwithin{example}{section}
\numberwithin{remark}{section}
\numberwithin{figure}{section}

\def\es{\varnothing}
\def\SB{\subseteq}

\def\aa{\alpha}
\def\bb{\beta}
\def\gg{\gamma}

\def\ee{\epsilon}

\def\b0{\boldsymbol 0}

\def\Ree{\mathbb R}

\def\imp{\Rightarrow}
\def\EQ{\Longleftrightarrow}

\def\AAA{{\cal A}}
\def\FFF{{\cal F}}

\def\DDD{{\cal D}}
\def\GGG{{\cal G}}
\def\HHH{{\cal H}}

\def\OOO{{\cal O}}

\def\RRR{{\cal R}}

\def\XXX{{\cal X}}

\def\begeq{\begin{equation}}
\def\endeq{\end{equation}}

\def\roster{\begin{enumerate}}
\def\endroster{\end{enumerate}}

\usepackage{graphicx}
\graphicspath{{converted_graphics/}}
\begin{document}

\title{Discrete piecewise linear functions}
\author{Sergei~Ovchinnikov \\
Mathematics Department\\
San Francisco State University\\
San Francisco, CA 94132\\
sergei@sfsu.edu} 

\date{\today} 
\maketitle

\begin{abstract}\noindent
The concept of permutograph is introduced and properties of integral functions on permutographs are established. The central result characterizes the class of integral functions that are representable as lattice polynomials. This result is used to establish lattice polynomial representations of piecewise linear functions on $\Ree^d$ and continuous selectors on linear orders.
\end{abstract}

\section{Introduction} \label{intro}

We begin with a simple yet instructive example. Let $f$ be a piecewise linear function (PL-function) defined by

$$
f(x)=\begin{cases}
	g_1(x), &\text{for $x\le-1$},\\
	g_2(x), &\text{for $-1\le x\le 1$},\\
	g_3(x), &\text{for $x\ge 1$}.
\end{cases}
$$
where
$$
g_1(x)=x+2,\quad g_2(x)=-x,\quad g_3(x)=0.5x-1.5.
$$
The graph of this function is shown in Figure~\ref{PL-graph}. 

\begin{figure}[h]
\centerline{\includegraphics{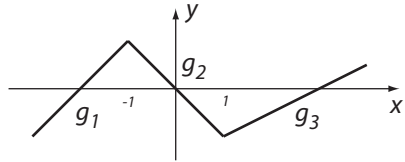}}
\caption{\small Graph of a PL-function.} \label{PL-graph}
\end{figure}

\noindent
It is easy to verify that
$$
f=g_1\wedge(g_2\vee g_3)=(g_1\wedge g_2)\vee(g_1\wedge g_3).
$$
We use the notations
$$
a\wedge b=\min\{a,b\}\quad\text{and}\quad a\vee b=\max\{a,b\}.
$$
throughout the paper. Thus, the function $f$ can be represented as a lattice polynomial in variables $g_1$, $g_2$, and $g_3$. This is true in general: any continuous PL-function $h$ on a convex domain in $\Ree^d$ is a lattice polynomial whose variables are linear `components' of $h$ (Theorem~\ref{PL theorem}). In various forms this result was independently established in~\cite{sB95},~\cite{sO02}, and~\cite{rW63} (however, see comments in Section~\ref{conclusion}, item 1).

The aim of the paper is to show that this result is essentially combinatorial. We introduce a class of functions on permutographs that we call `discrete piecewise linear functions', and show that these functions are representable as lattice polynomials. The discretization of the original problem is achieved by replacing the `continuity' and `linearity' properties of PL-functions by the `separation' and `linear ordering' properties of integral functions on permutographs.

Permutographs are isometric subgraphs of a weighted Cayley graph of the symmetric group; they are introduced in Section~\ref{permutographs}. In Section~\ref{functions on permutations}, we characterize lattice polynomials on permutographs in terms of the `separation property' and as `discrete PL-functions'. These characterizations are used in Section~\ref{PL-functions} to establish lattice polynomial representations of PL-functions on convex domains in $\Ree^d$.

In a totally different setting, the results of Section~\ref{functions on permutations} are used in Section~\ref{selectors} to obtain a polynomial representation for functions on linear orders. Topological properties of linear orders that are used in Section~\ref{selectors} are introduced in Appendix~\ref{appendix}. Some relevant topics are discussed in Section~\ref{conclusion}.

\section{Permutographs} \label{permutographs}

Let $X$ be a linearly ordered finite set of cardinality $n\ge 1$. We assume that $X$ is the set $\{1,\ldots,n\}$ ordered by the usual relation $<$. A {\em permutation} ({\em of order $n$}) is a bijection $\aa:X\to X$. We write permutations on the right, that is, $x\aa$ is the image of $x$ under $\aa$, compose them left to right (cf.~\cite{pC99,pC65}), and use the notation
$$
\aa=(x_1\cdots x_k\cdots x_n),
$$
where $x_k=k\aa$. For a given permutation $\aa=(x_1\cdots x_n)$, the elements of $X$ are linearly ordered by the relation $<_\aa$ defined by
$$
x_i<_\aa x_j\qquad\EQ\qquad i<j.
$$
In other words,
$$
x<_\aa y\qquad\EQ\qquad x\aa^{-1}<y\aa^{-1}.
$$
We write $x\le_\aa y$ if $x<_\aa y$ or $x=y$. Symbols $>_\aa$ and $\ge_\aa$ stand for the respective inverse relations.

A pair $\{x,y\}$ of elements of $X$ is called an {\em inversion} for a pair of permutations $\{\aa,\bb\}$ if $x$ and $y$ appear in reverse order in $\aa$ and $\bb$. The {\em distance} $d(\aa,\bb)$ between permutations $\aa$ and $\bb$ is defined as the number of inversions for the pair $\{\aa,\bb\}$. This distance equals one half of the cardinality of the symmetric difference of the binary relations $<_\aa$ and $<_\bb$. We say that a permutation $\gg$ {\em lies between} permutations $\aa$ and $\bb$ if
$$
d(\aa,\gg)+d(\gg,\bb)=d(\aa,\bb).
$$
It is straightforward to see that $\gg$ lies between $\aa$ and $\bb$ if and only if
$$
(\text{$x<_\aa y$ and $x<_\bb y$})\quad\imp\quad x<_\gg y\quad\text{for all $x,y\in X$.}
$$

The set of all permutations of $X$ forms the {\em symmetric group $S_n$} with the operation of composition and the identity element $\ee=(1\cdots n)$.

A partition $\pi=(X_1,\ldots,X_m)$ of the ordered set $(X,<)$ into a family of nonempty subsets is said to be an {\em ordered partition} if
$$
(x\in X_i,\;y\in X_j,\;i<j)\quad\imp\quad x<y.
$$
The ordered partition $(\{1\},\ldots,\{n\})$ is said to be {\em trivial}. 

For a given nontrivial ordered partition $\pi=(X_1,\ldots,X_m)$, the permutation $\tau_\pi$ reverses the order of elements in every set $X_i$. For instance,
$$
\tau_\pi=(1\,432\,5\,6\,87)
$$
for the ordered partition $\pi=(\{1\},\{2,3,4\},\{5\},\{6\},\{7,8\})$. 
Two permutations $\aa$ and $\bb$ are {\em $\pi$-adjacent} if $\aa\bb^{-1}=\tau_\pi$ for a nontrivial ordered partition $\pi$. They are {\em adjacent} if they are $\pi$-adjacent for some $\pi$. 
The adjacency relation on $S_n$ is symmetric and irreflexive. It defines a Cayley graph on $S_n$ that we denote by $\Upsilon_n$. By definition, the {\em weight} of an edge $\aa\bb$ is the distance $d(\aa,\bb)$ between permutations $\aa$ and $\bb$. We call the weighted graph $(\Upsilon_n,d)$ the {\em big permutograph on $S_n$}. This graph is $k$-regular for $k=2^{n-1}-1$. By definition, a {\em permutograph on $S_n$} (cf.~\cite{gT82}) is an isometric weighted subgraph of the big permutograph.

\begin{figure}[h]
\centerline{\includegraphics{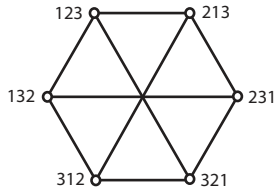}}
\caption{\small Big permutograph on $S_3$. The weights of edges are not shown.} \label{big3}
\end{figure}

\begin{example}
{\rm The graph $\Upsilon_3$ is shown in Figure~\ref{big3}. It is the complete bipartite graph $K_{3,3}$.
}
\end{example}

\begin{example} \label{permutohedron}
{\rm The graph of the permutohedron $\Pi_{n-1}$ is a permutograph on $S_n$. (See \cite{cB71,aB99,gT82,gZ95}; the term ``permutohedron'' was coined by Guilbaud and Rosenstiehl~\cite{gG63} in 1963.) This graph is a spanning graph of $\Upsilon_n$.
The edges of the big permutograph on $S_n$ link the opposite vertices of the faces of the permutohedron (see Figure~\ref{big4}).
}
\end{example}

\begin{figure}[h]
\centerline{\includegraphics{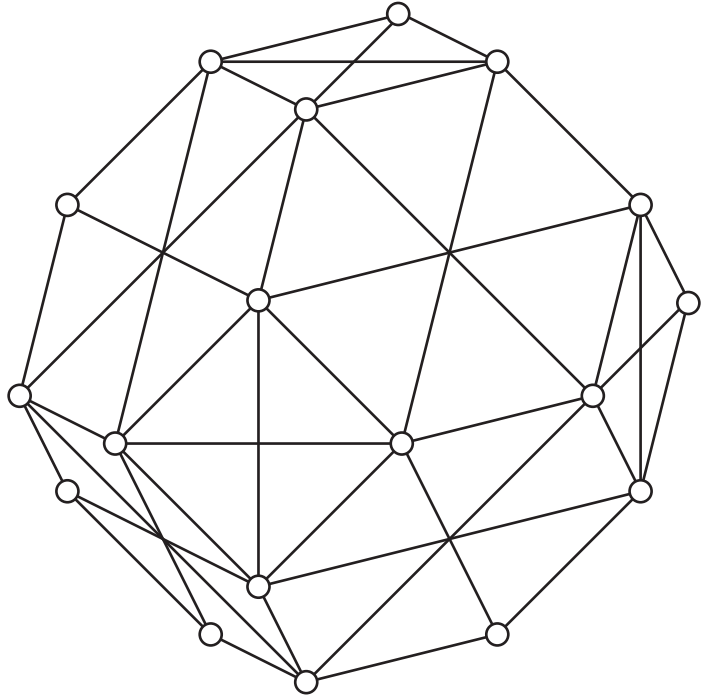}}
\caption{\small Big permutograph on $S_4$. The graph is drawn on the permutohedron $\Pi_3$. Only `visible' edges are shown. The weights of edges are not shown.} \label{big4}
\end{figure}

\section{Functions on permutations} \label{functions on permutations}

Let $S$ be a subset of $S_n$. We denote $\FFF(S)$ the set of all functions from $S$ to $X$. The ordering $<$ of the set $X$ induces a partial order on $\FFF(S)$:
$$
F<G\quad\EQ\quad F(\aa)<_\aa G(\aa)\quad\text{for all $F,G\in\FFF(S)$ and $\aa\in S$}.
$$
The poset $\FFF(S)$ is a complete distributive lattice with meet and join operations defined pointwise by
\begeq \label{lattice ops}
(F\wedge G)(\aa)=F(\aa)\wedge G(\aa)\quad\text{and}\quad (F\vee G)(\aa)=F(\aa)\vee G(\aa),
\endeq
respectively. In the right-hand sides of the equations in~\ref{lattice ops}, the meet and join operations are defined with respect to the linear ordering $\le_\aa$. We use this convention throughout the paper.

For a given $k\in X$ we define
\begeq \label{constant}
G_k(\aa)=k\quad\text{for all $\aa\in S$},
\endeq
the {\em constant function} on the set $S$. 

\begin{example} \label{statistic}
{\rm The $k$th {\em order statistic} $M_k$ on $S$ is defined by
$$
M_k(\aa)=x_k\quad\text{for $\aa=(x_1\cdots x_n)\in S$}.
$$
Let $\XXX_k$ be a family of subsets of $X$ defined by
$$
\XXX_k=\{Y\SB X: |X\setminus Y|=k-1\}.
$$
We have the following formula for the order statistic $M_k$ (cf.~\cite{sO96}):
\begeq \label{M_k}
M_k = \bigvee_{Y\in\XXX_k}\bigwedge_{j\in Y} G_j
\endeq
Indeed, it suffices to note that $x_1<_\aa\cdots<_\aa x_n$ for $\aa=(x_1\cdots x_n)$, so the maximum in~(\ref{M_k}) is attained at $Y=\{x_k,\ldots,x_n\}$.
}
\end{example}

The right-hand side of the equation in~(\ref{M_k}) is a lattice polynomial written in its disjunctive normal form. Since $\FFF(S)$ is a distributive lattice, any lattice polynomial in variables $G_j$'s can be written in the disjunctive normal form~\cite{gB79}. This fact motivates the following definition.

\begin{definition}
{\rm Let $\{K_i\}_{i\in I}$ be a family of subsets of the set $X=\{1,\ldots,n\}$. A {\em polynomial} on $S$ is a function $F:S\to X$ defined by
\begeq \label{poly}
F(\aa)=\bigvee_{i\in I}\bigwedge_{j\in K_i} G_j(\aa).
\endeq
Note that we may assume that the family $\{K_i\}_{i\in I}$ is an antichain in the lattice $2^X$ of all subsets of $X$.
}
\end{definition}

In this section we give two characterizations of polynomial functions on permutations.

\begin{definition}
{\rm Let $S$ be a nonempty subset of $S_n$. A function $F:S\to X$ satisfies the {\em separation property} ({\em S-property}) if, for any $\aa,\bb\in S$, there is $u\in X$ such that
$$
u\le_\aa F(\aa)\quad\text{and}\quad u\ge_\bb F(\bb).
$$
}
\end{definition}

\begin{theorem} \label{S=poly}
A function $F:S\to X$ satisfies the S-property if and only if it is a polynomial.
\end{theorem}

\begin{proof}
(Necessity.) Suppose that $F$ satisfies the S-property. For $\gg\in S$, let $K_\gg=\{v\in X\mid v\ge_\gg F(\gg)\}$. By the S-property, there is $u\in X$ such that $u\le_\aa F(\aa)$ and $u\ge_\gg F(\gg)$, so $u\in K_\gg$. Since $u\le_\aa F(\aa)$, we have
$$
\bigwedge_{j\in K_\gg} G_j(\aa)\le_\aa F(\aa)\qquad\text{for every $\gg\in S$.}
$$
Clearly,
$$
\bigwedge_{j\in K_\aa} G_j(\aa) = F(\aa).
$$
Thus,
$$
\bigvee_{\gg\in S}\bigwedge_{j\in K_\gg} G_j(\aa)=F(\aa),
$$
that is, $F$ is a polynomial.

(Sufficiency.) Suppose that $F$ is a polynomial in the form~(\ref{poly}). Let $\aa,\bb\in S$ and suppose that $u<_\bb F(\bb)$ for any $u\le_\aa F(\aa)$. By~(\ref{poly}), for every $i$ there is $j_i\in K_i$ such that $j_i\le_\aa F(\aa)$ (recall that $G_j$'s are constant functions). By our assumption, $j_i<_\bb F(\bb)$. Since $j_i\in K_i$, we have 
$$
\bigwedge_{j\in K_i} G_j(\bb)<_\bb F(\bb).
$$
Therefore,
$$
\bigvee_{i\in J}\bigwedge_{j\in K_i} G_j(\bb)<_\bb F(\bb),
$$
contradicting~(\ref{poly}). It follows that $F$ satisfies the S-property.
\end{proof}


\begin{definition}
{\rm A function $F:S\to X$ is said to be a {\em discrete piecewise linear function} (a {\em DPL-function}) if, for any two $\pi$-adjacent permutations $\aa,\bb\in S$ with $\pi=(X_1,\ldots,X_m)$, we have
\begeq \label{DPL}
F(\aa)\in X_i\aa\quad\text{and}\quad F(\bb)\in X_i\bb\quad\text{for some $1\le i\le m$}.
\endeq
See Figure~\ref{DPL-function}.
}
\end{definition}

\begin{figure}[h]
\centerline{\includegraphics{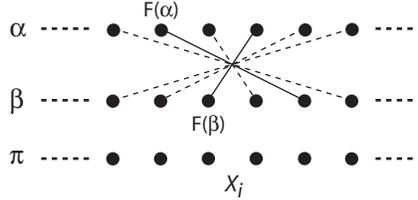}}
\caption{\small Values of a DPL-function on two adjacent permutations.} \label{DPL-function}
\end{figure}

\begin{theorem} \label{DPL=S}
Let $S$ be the vertex set of a permutograph on $S_n$. A function $F$ on $S$ is a DPL-function if and only if it satisfies the S-property.
\end{theorem}

\begin{proof}
(Necessity.) Let $F$ be a DPL-function on $S$. We need to show that for given $\aa,\bb\in S$ there is $u\in X$ such that
$$
u\le_\aa F(\aa)\quad\text{and}\quad u\ge_\bb F(\bb).
$$
If $F(\bb)\le_\bb F(\aa)$, we may choose $u=F(\aa)$. Thus, in what follows, we assume that
\begeq \label{assume}
F(\aa)<_\bb F(\bb).
\endeq

The proof is by induction on the length $k$ of a shortest $\aa\bb$-path in $S$. 

For $k=1$, the permutations $\aa$ and $\bb$ are adjacent and the result follows immediately from the definition of a DPL-function: just let $u$ be the maximum element in $X_i\bb$ (see~(\ref{DPL}) and recall that elements of $X_i\aa$ and $X_i\bb$ are in reverse order).

For the inductive step, suppose that 
$$
\aa=\aa_0,\gg=\aa_1,\ldots,\aa_k=\bb
$$
is a shortest $\aa\bb$-path of length $k>1$ in $S$. By the induction hypothesis, there is $u\in X$ such that
$$
u\ge_\bb F(\bb)\quad\text{and}\quad u\le_\gg F(\gg).
$$
Since $\aa=(x_1,\ldots,x_n)$ and $\gg=(y_1,\ldots,y_n)$ are $\pi$-adjacent for some ordered partition $\pi=(X_1,\ldots,X_m)$, we have $F(\aa)\in X_i\aa$ and $F(\gg)\in X_i\gg$ for some $1\le i\le m$, that is,
$$
x_p\le_\aa F(\aa)\le_\aa x_q\quad\text{and}\quad y_p\le_\gg F(\gg)\le_\gg y_q,
$$
where $p$ and $q$ are the minimum and maximum elements of $X_i$, respectively.

If $u<_\gg y_p$, then $u<_\aa x_p$, because $\pi$ is an ordered partition. Therefore, $u<_\aa F(\aa)$, since $x_p\le_\aa F(\aa)$.

Otherwise, $u\in X_i\gg$, implying $u\in X_i\aa$, since $\aa$ and $\gg$ are $\pi$-adjacent. Since $\gg$ lies between $\aa$ and $\bb$, and $F(\aa)\in X_i\gg$, any element of $X_i\aa$ which is greater then $F(\aa)$ in the linear ordering $(X_i,<_\aa)$ must be less than $F(\aa)$ in the linear ordering $(X,<_\bb)$. Hence, if $u>_\aa F(\aa)$, then we must have $F(\bb)\le_\bb u<_\bb F(\aa)$, in contradiction with our assumption in~(\ref{assume}). It follows that $u\le_\aa F(\aa)$.

(Sufficiency.) We assume that $F$ satisfies the S-property. Let $\aa$ and $\bb$ be two $\pi$-adjacent permutations in $S$ with $\pi=(X_1,\ldots,X_m)$ and suppose that $F(\aa)\in X_i\aa$ and $F(\bb)\in X_j\bb$ with $i<j$. Clearly, $u\ge_\bb F(\bb)$ implies $u>_\aa F(\aa)$, contradicting the S-property. Hence, $i\ge j$. By symmetry, $j=i$. Therefore $F$ is a DPL-function.
\end{proof}

The following theorem summarizes the results of Theorems~\ref{S=poly} and~\ref{DPL=S}.

\begin{theorem} \label{DPL representation}
Let $S$ be a permutograph on $S_n$ and $F$ be a function on $S$ with values in $X$. The following statements are equivalent:
\roster
	\item[{\rm(i)}] $F$ is a discrete piecewise linear function.
	\item[{\rm(ii)}] $F$ satisfies the separation property.
	\item[{\rm(iii)}] $F$ is a polynomial.
\endroster
\end{theorem}

\section{Piecewise linear functions on $\Ree^d$} \label{PL-functions}

A {\em closed domain} in $\Ree^d$ is the closure of a nonempty open set in $\Ree^d$. In this section, $D$ is a convex closed domain in $\Ree^d$ and $\{g_i(x)\}_{1\le i\le n}$ is a family of distinct (affine) linear functions on $D$. We assume that 
$$
\text{int}(D)\cap\text{ker}(g_i-g_j)\neq\es
$$
for at least one pair of distinct functions $g_i$ and $g_j$, where $\text{int}(D)$ stands for the interior of $D$.

Let $\HHH$ be the arrangement of all distinct hyperplanes in $\Ree^d$ that are solutions of the equations in the form $g_i(x)=g_j(x)$ and have nonempty intersections with $\text{int}(D)$. We denote by $\RRR$ the family of nonempty intersections of the regions of $\HHH$ with $\text{int}(D)$ and use the same name `region' for elements of $\RRR$. The {\em region graph} $\GGG$ of the arrangement $\HHH$ has $\RRR$ as the set of vertices; the edges of the graph are pairs of adjacent regions.

It is easy to see that the functions $g_1,\ldots,g_n$ are linearly ordered over any region in $\RRR$, that is, for a given $R\in\RRR$ there is a permutation $(i_1\cdots i_n)$ such that
$$
g_{i_1}(x)<\cdots<g_{i_n}(x)\qquad\text{for all $x\in R$.}
$$
This correspondence defines a mapping $\varphi:\RRR\to S_n$. We treat the graph $\GGG$ as a weighted graph: the weight of an edge $PQ$ is the distance between permutations $\varphi(P)$ and $\varphi(Q)$.

\begin{theorem} \label{G is permutograph}
The mapping $\varphi$ defines an isometric embedding of the weighted region graph $\GGG$ into the big permutograph $\Upsilon_n$. Thus the image $\varphi(\GGG)$ of the region graph is a permutograph.
\end{theorem}

\begin{proof}
First we show that $\varphi$ is a one-to-one function. Let $P$ and $Q$ be two distinct regions in $\RRR$. Let $H$ be a hyperplane in $\HHH$ separating $P$ and $Q$. The hyperplane $H$ is the solution set of some equation $g_i=g_j$ with $i\neq j$. Therefore the functions $g_i$ and $g_j$ are in reverse order over regions $P$ and $Q$. It follows that the permutations $\varphi(P)$ and $\varphi(Q)$ are distinct.

Now we show that the permutations corresponding to two adjacent regions $P,Q\in\RRR$ are adjacent vertices in the graph $\Upsilon_n$. Let $H\in\HHH$ be the hyperplane separating $P$ and $Q$ and $F\SB H$ be the relative interior of the common facet of $P$ and $Q$. It is clear that two components $g_i$ and $g_j$ are in the reverse order over $P$ and $Q$ if and only if $\text{ker}(g_i-g_j)=H$. It follows that the values $g_i(x)$ of the components over $F$ define the same ordered partition $\pi$ of the set $X=\{1,\ldots,n\}$ for every $x\in F$. For this partition $\pi$, the permutations $\varphi(P)$ and $\varphi(Q)$ are $\pi$-adjacent.

To complete the proof we need to prove that the image of $\GGG$ under $\varphi$ is an isometric subgraph of $\Upsilon_n$. Let $P$ and $Q$ be two distinct regions in $\RRR$. A simple topological argument (cf.~\cite{sO05}; convexity of the domain is essential in this argument) shows that there are points $p\in P$ and $q\in Q$ such that the line segment $[p,q]$ does not intersect cells of $\HHH$ of dimension less than $d-1$. The regions with nonempty intersections with $[p,q]$ form a path $R_0=P,R_1,\ldots,R_m=Q$ in $\GGG$. Since functions $g_i$'s are linear, the number of inversions for the pair $\{\varphi(P),\varphi(Q)\}$ equals the total number of inversions corresponding to the pairs of adjacent regions in the path. Since the number of inversions for two adjacent regions is the weight of the edge joining these regions, the length of the path $\varphi(R_0),\ldots,\varphi(R_m)$ equals the distance between $\varphi(P)$ and $\varphi(Q)$, that is, $\varphi$ is an isometric embedding.
\end{proof}

\begin{definition}
{\rm Let $D$ be a convex closed domain in $\Ree^d$. A function $f:D\to\Ree$ is said to be a {\em piecewise linear function} ({\em PL-function}) if there is a finite family $\DDD$ of closed domains such that $D=\cup\,\DDD$ and $f$ is (affine) linear on every domain in $\DDD$. A linear function $g$ such that $g|_R=f|_R$ for some domain $R\in\DDD$ is said to be a {\em component} of $f$.
}
\end{definition}

Note that in applied papers (see, for instance,~\cite{jT99,rW63} and references thereof) the domain $D$ is a convex polyhedron in $\Ree^d$.

Clearly, a PL-function on $D$ is continuous. Let $f$ be a PL-function on a given convex closed domain $D\SB\Ree^d$ and $\{g_1,\ldots,g_n\}$ be the set of components of $f$.

Let us define a function $F(\aa)$ on the set of vertices of the permutograph $\varphi(\GGG)$ as follows:
$$
F(\aa)=G_i(\aa)\qquad\EQ\qquad f(x)=g_i(x)\quad\text{for $x\in\varphi^{-1}(\aa)$.}
$$
Because $f$ is a continuous function, the function $F$ is a DPL-function on $\varphi(\GGG)$. By Theorem~\ref{DPL representation}, there is a family $\{K_i\}_{i\in I}$ such that
$$
F(\aa)=\bigvee_{i\in I}\bigwedge_{j\in K_i} G_j(\aa).
$$
The functions $g_i$'s are ordered over a given region $R$ as functions $G_i$'s are ordered with respect to the relation $<_\aa$ for $\aa=\varphi(R)$. Therefore we have the following theorem (Theorem 2.1 in~\cite{sO02}).

\begin{theorem} \label{PL theorem}
Let $f$ be a PL-function on a convex closed domain $D$ in $\Ree^d$ and $\{g_1,\ldots,g_n\}$ be the set of components of $f$. There is a family $\{K_i\}_{i\in I}$ of subsets of the set $X=\{1,\ldots,n\}$ such that
\begeq \label{PL representation}
f(x)=\bigvee_{i\in I}\bigwedge_{j\in K_i} g_j(x),\qquad\text{for $x\in D$.}
\endeq
The converse is also true: Let $\{g_1,\ldots,g_n\}$ be a family of affine linear functions on $D$. Then a function in the form~{\rm(\ref{PL representation})} is a PL-function.
\end{theorem}

The following simple corollary is of importance in some applications (see~\cite{sB95,dM86,sW04,cW05}):

\begin{corollary}
A PL-function is representable as a difference of two concave (equivalently, convex) PL-functions.
\end{corollary}

\begin{proof}
Let $f$ be a PL-function in the form~(\ref{PL representation}) and let $h_i(x)=\bigwedge_{j\in K_i} g_j(x)$. Note that $h_i$'s are concave functions. Since
$$
h_i = \sum_k h_k-\sum_{k\neq i}h_k,
$$
we have
$$
f=\bigvee_{i\in I}h_i=\sum_k h_k-\bigwedge_{i\in I}\sum_{k\neq i}h_k
$$
Because sums and minimums of concave functions are concave, we have the desired representation.
\end{proof}

Clearly, a PL-function $f$ on $D$ is a `selector' of its components $g_j$'s, that is, for any $x\in D$ there is $i$ such that $f(x)=g_i(x)$ (cf.~Section~\ref{selectors}). Conversely, let $f$ be a continuous selector of a family of linear functions $\{g_1,\ldots,g_n\}$ and let $R$ be a region of the arrangement $\HHH$ defined by this family over $D$. The functions $g_i$'s are linearly ordered over $R$. Since $f$ is a continuous function and $R$ is connected, we must have $f=g_i$ over $R$ for some index $i$. It follows that a continuous selector $f$ is a PL-function on $D$ and therefore admits a polynomial representation~(\ref{PL representation}).  This case is of interest in the ``nonsmooth critical point theory''~(see~\cite{sB95}).

\section{Selectors and invariant functions} \label{selectors}

Let $X$ be an arbitrary set and $D$ be a subset of $X^d$. Let $\{g_1,\ldots,g_n\}$ be a family of functions on $D$ with values in $X$. A function $f:D\to X$ is said to be a {\em selector of the functions $g_i$'s} if for any $x\in D$ there is $i$ such that $f(x)=g_i(x)$. A {\em coordinate selector} (cf.~\cite{vM07}) is a selector of the coordinate functions $g_i(x)=x_i$ for $1\le i\le d$. 

In the rest of this section we assume that $X$ is a linearly ordered set endowed with interval topology (see Appendix~\ref{appendix} for notations and relevant results).

Suppose that $X$ is a connected space and let $f$ be a continuous coordinate selector on $X^d$. For a given permutation $\aa\in S_d$, the sets
$$
A_i=\{x\in\OOO_\aa : f(x)=x_i\},\qquad 1\le i\le d,
$$
are closed disjoint sets and the chamber $\OOO_\aa$ is their finite union. Since $\OOO_\aa$ is a connected set (Theorem~\ref{connected}), we must have $f(x)=x_k$ on $\OOO_\aa$ for some $1\le k\le d$. We define a function $F$ on the permutohedron graph $\Pi_{d-1}$ by letting $F(\aa)=k$ if $f(x)=x_k$ on $\OOO_\aa$. By Theorem~\ref{chambers}(2), $F$ is a DPL-function on the permutograph $\Pi_{d-1}$. By Theorems~\ref{DPL representation} and~\ref{chambers}(1), we have the following result:

\begin{theorem} \label{selector=polynom}
Let $X$ be a connected linear order. A function $f:X^d\to X$ is a continuous coordinate selector if and only if it is a lattice polynomial in variables $x_1,\ldots,x_d$.
\end{theorem}

Note that the result of this theorem does not hold for disconnected linear orders. Indeed, let $X=U\cup V$ where $U$ and $V$ are nonempty disjoint open sets and let us define $f:X^d\to X$ by
$$
f(x_1,\ldots,x_d)=\begin{cases}
	x_1, &\text{if $x_1\in U$,}\\
	x_2, &\text{if $x_1\in V$.}
\end{cases}
$$
The function $f$ is a continuous selector which is not representable as a lattice polynomial.

Let $X$ be a linear order and let $f$ be a lattice polynomial in variables $x_1,\ldots,x_d$, that is,
\begeq \label{selector poly}
f(x)=\bigvee_{i\in I}\bigwedge_{j\in K_i} x_j,\qquad\text{for $x=(x_1,\ldots,x_d)\in X^d$,}
\endeq
where $\{K_i\}_{i\in I}$ is a family of subsets of the set $\{1,\ldots,d\}$. It is clear that $f(x\psi)=f(x)\psi$ for any automorphism $\psi\in\AAA(X)$ (cf.~Appendix~\ref{appendix}). We show below that the converse is true for a special class of linear orders.

In the rest of this section, $X$ is a doubly homogeneous linear order, that is, there is 
a doubly transitive $\ell$-permutation group $G$ acting on $X$. By Theorem~\ref{homogeneity}, $G$ is $m$-transitive for all $m\ge 2$. 

A function $f:X^d\to X$ is said to be {\em invariant} (under actions from $G$) if 
$$
f(x_1\psi,\ldots,x_d\psi)=f(x_1,\ldots,x_d)\psi
$$
for all $(x_1,\ldots,x_d)\in\Ree^d$ and $\psi\in G$.

\begin{theorem} \label{invariant->selector}
An invariant function $f:X^d\to X$ is a coordinate selector.
\end{theorem}

\begin{proof}
Let $f$ be an invariant function on $X^d$. If $y=f(x_1,\ldots,x_d)$, then we have
$$
y=f(x_1,\ldots,x_d)=f(x_1\psi,\ldots,x_d\psi)=f(x_1,\ldots,x_d)\psi=y\psi,
$$
for any automorphism $\psi\in G$ that fixes elements $x_1,\ldots,x_d$. Suppose that $y\neq x_i$ for all $1\le i\le d$. Since $X$ is $(d+1)$-homogeneous, there is an automorphism in $G$ that fixes elements $x_1,\ldots,x_d$ and such that $y\psi\neq y$, a contradiction. Therefore, $f(x_1,\ldots,x_d)\in\{x_1,\ldots,x_d\}$, that is, $f$ is a coordinate selector.
\end{proof}

Note that the linear order $X$ is not necessarily connected, so we cannot simply apply the result of Theorem~\ref{selector=polynom} to show that an invariant function is a polynomial. However this result holds as the following argument demonstrates.

Since $X$ is $d$-homogeneous and the coordinates appear in the same order for all points in a chamber $\OOO_\aa$, we must have (by Theorem~\ref{invariant->selector}) $f(x)=x_k$ on $\OOO_\aa$ for some $1\le k\le d$. As in the previous section, we define a function $F$ on the permutohedron graph $\Pi_{d-1}$ by letting $F(\aa)=k$ if $f(x)=x_k$ on $\OOO_\aa$. By Theorem~\ref{chambers}(2), $F$ is a DPL-function on the permutograph $\Pi_{d-1}$. By Theorems~\ref{DPL representation} and~\ref{chambers}(1), we have the following result~\cite{sO98}:

\begin{theorem}
Let $X$ be a doubly homogeneous linear order. A continuous function $f:X^d\to X$ is invariant if and only if it is a lattice polynomial in variables $x_1,\ldots,x_d$.
\end{theorem}

We define the $k$th {\em order statistic} (cf.~Example~\ref{statistic}) $x^{(k)}$ on $X^d$ by arranging a $d$-tuple $(x_1,\ldots,x_d)$ in the increasing order:
$$
x^{(1)}\le\cdots\le x^{(k)}\le\cdots\le x^{(d)}.
$$
Clearly, an order statistic is a symmetric, continuous, and invariant function on $X^d$. The converse is also true~\cite{sO96}:

\begin{theorem}
Let $f$ be a symmetric continuous invariant function on $X^d$. Then $f$ is an order statistic.
\end{theorem}

\begin{proof}
For a given sequence $a_1<\cdots<a_d$, we have $f(a_1,\ldots,a_d)=a_k$ for some $1\le k\le d$, by Theorem~\ref{invariant->selector}. Suppose that $x_i$'s are distinct elements of $X$. Since $f$ is $d$-homogeneous, there is $\psi\in G$ such that $a_i\psi=x^{(i)}$ for all $1\le i\le d$. Because $f$ is symmetric and invariant, we have
\begin{align*}
f(x_1,\ldots,x_d)&=f(x^{(1)},\ldots,x^{(d)})=f(a_1\psi,\ldots,a_2\psi)\\
&=f(a_1,\ldots,a_d)\psi=a_k\psi=x^{(k)}.
\end{align*}
Therefore, $f$ is the $k$th order statistics over chambers in $X^d$. The result follows from Theorem~\ref{chambers}(1), since $f$ is a continuous function.
\end{proof}

As in Example~\ref{statistic}, we have the following lattice polynomial representation for the $k$th order statistics:
$$
x^{(k)} = \bigvee_{Y\in X_k}\bigwedge_{j\in Y} x_j,
$$
where $X_k$ is the family of $(d-k+1)$-element subsets of $\{1,\ldots,d\}$. Clearly, $x^{(1)}=\min\{x_1,\ldots,x_d\}$ and $x^{(d)}=\max\{x_1,\ldots,x_d\}$.

\section{Concluding remarks} \label{conclusion}

\newcounter{myrule}
\begin{list}{\arabic{myrule}.}{\setlength{\leftmargin}{12pt}}
\usecounter{myrule}
	\item The statement that a continuous piecewise linear function admits a representation as a max-min composition of its linear components appears to be intuitively clear. Apparently, this result was first stated in~\cite{rW63} and repeated in~\cite{jT99}. As it happens in applied areas, these publications lack precise definitions and assumptions. For instance, the result of Theorem~\ref{PL theorem} does not hold for non-convex domains (see the next remark), but this condition is not used in the `proofs' found in~\cite{jT99}. In a different context, this result appears as Corollary~2.1 in~\cite{sB95}, but again the proof is unsatisfactory. In its present form, the result was formulated and proven independently in~\cite{sO02}. A multidimensional analog of Theorem~\ref{PL theorem} is also found there.
	\item The convexity of the domain $D$ in Theorem~\ref{PL theorem} is an essential assumption. Consider, for instance, the domain
$$
D=\{(x,y)\in\Ree^2 : y\le |x|\}
$$
in $\Ree^2$ and define a PL-function $f$ on $D$ by
$$
f(x,y)=\begin{cases}
	y, &\text{if $x\wedge y\ge 0$,}\\
	0, &\text{otherwise.}
\end{cases}
$$
It is clear that $f$ is not representable as a lattice polynomial in terms of functions $g_1(x,y)=0$ and $g_2(x,y)=y$.

	\item The result of Theorem~\ref{PL theorem} does not hold in general for piecewise polynomial functions. For instance, the function
$$
h(x)=\begin{cases}
	0, &\text{if $x<0$,}\\
	x^2, &\text{if $x\ge 0$,}
\end{cases}
$$
cannot be expressed by means of minimum and maximum operations on the zero function and $x^2$ (polynomial `components' of $h$). On the other hand, we clearly have
$$
h(x)=((x^3+x)\vee 0)\wedge x^2,
$$
that is, $h$ is definable with polynomials by means of the operations $\wedge$ and $\vee$. The ``Pierce-Birkhoff conjecture'' states that any continuous piecewise polynomial function on $\Ree^d$ can be obtained from the polynomial ring $R[x_1,\ldots,x_d]$ by iterating the operations $\wedge$ and $\vee$. (The problem is still open; see~\cite{gB56,mH62,lM84,mS06}).
	\item It is not difficult to show that selectors in the form~(\ref{selector poly}) are in one-to-one correspondence with nonempty antichains of subsets of the set $\{1,\ldots,d\}$. Thus the total number of continuous selectors on $X^d$ is the Dedekind number (entry A007153 in~\cite{nS}).

A more involved problem is counting functions on $X^d$ that can be expressed as lattice polynomials using the operations $\wedge$ and $\vee$, in which every variable appears exactly once. They are known as `read-once expressions' (cf.~\cite{yC07}) and of importance in PL Morse theory~\cite{vM07}. It can be shown that the number $M(d)$ of distinct read-once functions on $X^d$ equals twice the number of total partitions of $d$ and satisfies the reccurence relation
$$
M(n)=(n+1)M(n-1)+\sum_{k=2}^{n-2}\binom{n-1}{k}M(k)M(n-k),
$$
with initial conditions $M(0)=1$, $M(1)=1$, and $M(2)=2$ (cf.~\cite{rS99} and entry A000311 in~\cite{nS}).
\end{list}


\appendix
\section{Appendix: interval topology}\label{appendix}

Let $(X,<)$ be a linear order with $|X|>2$. We write $x\le y$ if $x<y$ or $x=y$ in $X$. An {\em open ray} in $X$ is a subset in the form
$$
(\leftarrow,a)=\{x\in X : x<a\}\quad\text{or}\quad (a,\to)=\{z\in Z : z>a\},\quad a\in X.
$$
An {\em open interval} in $X$ is either an open ray or a subset in the form
$$
(a,b)=\{x\in X : a<x<b\},\qquad \text{for $a<b$ in $X$}.
$$
{\em Closed intervals} $[a,b]$, $(\leftarrow,a]$,  and $[a,\to)$ are defined similarly. A {\em gap} in $X$ is an empty open interval $(a,b)$. The family of open intervals is a base for the {\em interval topology} (order topology) on $X$.

Let $\aa=(i_1\cdots i_d)$ be a permutation of order $d$. A {\em chamber} $\OOO_\aa$ is a subset of $X^d$ defined by
$$
\OOO_\aa=\{(x_1,\ldots,x_d) : x_{i_1}<\cdots<x_{i_d}\}.
$$
Two chambers $\OOO_\aa$ and $\OOO_\bb$ are {\em adjacent} if 
$$
\aa=(i_1\cdots i_k i_{k+1}\cdots i_d)\quad\text{and}\quad \bb=(i_1\cdots i_{k+1} i_k\cdots i_d).
$$

Note that for $X=\Ree$ the chambers in $\Ree^d$ are regions of the braid arrangement in $\Ree^d$.

\begin{theorem} \label{connected}
Let $X$ be a linear order which is connected in its interval topology. Then the chambers of $X^d$ are connected sets.
\end{theorem}

\begin{proof}
It suffices to show that the chamber $\OOO_\ee=\{x\in X^d : x_1<\cdots<x_d\}$ is a connected set. Indeed, any other chamber is an image of $\OOO_\ee$ under a homeomorphism of $X^d$ onto itself defined by a permutation of coordinates.

For any two points $(x_1,\ldots,x_d)$ and $(y_1,\ldots,y_d)$ in $X^d$ such that $x_i=y_i$ for all $i\neq k$, the `line segment'
$$
\{(z_1,\ldots,z_d):z_i=x_i\; (i\neq k),\; x_k\wedge y_k\le z_k\le x_k\vee y_k\}
$$
is connected. It is not difficult to see that for any two points in the chamber $\OOO_\ee$ there is a sequence of points in $\OOO_\ee$ such that consecutive points differ in exactly one coordinate. The union of corresponding `line segments' is a connected set (a `path') containing the two points. It follows that the chamber $\OOO_\ee$ is a connected set.
\end{proof}

The next theorem puts forth some properties of chambers established in~\cite{sO96} and~\cite{sO98}.

\begin{theorem} \label{chambers}
Let $X$ be a linear order without gaps. 

{\rm(1)} The chambers in $X^d$ are open sets and their union is dense in $X^d$.

{\rm(2)} Let $\OOO_{(i_1\cdots i_k i_{k+1}\cdots i_d)}$ and $\OOO_{(i_1\cdots i_{k+1} i_k\cdots i_d)}$ be two adjacent chambers in $X^d$ and $f:X^d\to X$ be a continuous function such that 
$$
f(x)=x_p\text{~for~}x\in\OOO_{(i_1\cdots i_k i_{k+1}\cdots i_d)}\text{~and~}f(x)=x_q \text{~for~} x\in\OOO_{(i_1\cdots i_{k+1} i_k\cdots i_d)}.
$$
Then one of the following holds:
\roster
	\item[{\rm(i)}] $p=q$,
	\item[{\rm(ii)}] $p=i_k$, $q=i_{k+1}$,
	\item[{\rm(iii)}] $p=i_{k+1}$, $q=i_k$.
\endroster
\end{theorem}

Clearly, the results of Theorem~\ref{chambers} hold for connected linear orders. Another class of linear orders without gaps consists of doubly homogeneous linear orders. For details about ordered permutation groups the reader is referred to~\cite{aG81}.

Let $X$ be a linear order and $\AAA(X)$ be the group of automorphisms (order-preserving permutations) of $X$. This group acquires the pointwise order from $X$, that is, $\aa<\bb$ if and only if $x\aa<x\bb$ for all $x\in X$. This order makes $\AAA(X)$ a lattice-ordered permutation group ($\ell$-permutation group), that is, $\AAA(X)$ is a lattice and the order is preserved by multiplication on both sides. The meet and join operations are also defined pointwise. A subgroup $G$ of $\AAA(X)$ which is also a sublattice is called an $\ell$-permutation group.

A subgroup $G\SB\AAA(X)$ is said to be $m$-transitive if for all 
$$
x_1\le\cdots\le x_m\quad\text{and}\quad y_1\le\cdots\le y_m
$$
in $X$, there exists $\aa\in G$ such that $x_i\aa=y_i$ for all $1\le i\le m$. If $G$ is $m$-transitive, we say that $X$ is $m$-homogeneous. It is clear, that a $2$-homogeneous (doubly homogeneous) linear order does not have gaps.

The result of the following theorem is Lemma~1.10.1 in~\cite{aG81}.

\begin{theorem} \label{homogeneity}
If $G\SB\AAA(X)$ is a doubly transitive $\ell$-permutation group, then it is $m$-transitive for all $m\ge 2$.
\end{theorem}

\end{document}